\documentclass[aap,  preprint]{imsart}

\usepackage{amsthm,amsmath,natbib}
\usepackage{amsfonts}

\startlocaldefs

\newtheorem{theorem}{Theorem}
\newtheorem{lemma}{Lemma}

\newtheorem{remark}{Remark}

\def\be{\begin{equation}}
\def\ee{\end{equation}}
\def\beq{\begin{eqnarray*}}
\def\eeq{\end{eqnarray*}}
\def\bea{\begin{eqnarray}}
\def\eea{\end{eqnarray}}

\def\Var{{\rm Var}}
\def\var{{\rm Var}}
\def\Cov{\rm Cov}

\newcommand{\BY}{\mathbf{Y}}

\newcommand{\reals}{{\mathbb R}}
\newcommand{\bbr}{\reals}
\newcommand{\bbz}{{\mathbb Z}}
\newcommand{\bbc}{{\mathbb C}}

\newcommand{\eid}{\overset{d}{=}}
\newcommand{\vep}{\varepsilon}

\newcommand{\one}{{\bf 1}}
\newcommand{\pp}{{\prime\prime}}

\endlocaldefs

\begin{document}

\begin{frontmatter}

\title{Clustering of large deviations in moving average processes: the short memory regime}
\runtitle{Clustering of large deviations}


\author{\fnms{Arijit} \snm{Chakrabarty}\ead[label=e1]{arijit.isi@gmail.com}}
\address{Theoretical Statistics and\\ Mathematics Unit \\
  Indian Statistical Institute, Kolkata \\
\printead{e1}}
\and
\author{\fnms{Gennady} \snm{Samorodnitsky} \thanksref{t1} \ead[label=e2]{gs18@cornell.edu}}
  \thankstext{t1}{ The corresponding author. Research partially 
   supported by NSF grant  DMS-2015242  and AFOSR grant
   FA9550-22-1-0091 at Cornell University.}
 \address{School of Operations Research\\
   and Information Engineering\\
  Cornell University\\
\printead{e2}}
\affiliation{Indian Statistical Institute, Kolkata and Cornell University}

\runauthor{Chakrabarty and Samorodnitsky}

\begin{abstract}
We describe the cluster of large deviations events that arise when one
such large deviations event occurs. We work in the framework of an
infinite moving average process with a noise that has finite
exponential moments. 
\end{abstract}



\end{frontmatter}


\section{Introduction} \label {sec:intro}
Equity premium is a common notion in finance. It is the difference
between long-term cumulative returns on the stock and the safe interest rate
returns (see e.g. \cite{goyal:welch:2003}); it enters as input into
stock options issued on a regular basis. Long-term persistence in
climate refers to the phenomenon where climatologically relevant
parameters (for example, the temperature) take higher than usual mean 
values over multiple time long periods (see
e.g. \cite{rybski:bunde:havlin:vonstorch:2006}). What these diverse
themes and many others like them have in common is that in the
critical cases one observes a time series over long intervals where it
averages out to ``unexpected'' values. These unexpected values cluster. What
type of probabilistic analysis 
should address such phenomena?

A very crude classification of how we analyse random systems might
split the work into distributional analysis and large deviations
analysis. The distributional analysis deals with the ``usual''
deviations of a system from its ``average'' state, while the large
deviations analysis deals the ``unusually large'' deviations, that
are, by necessity, rare (but may have a major impact). The
idea of clustering is a major idea in how we look at random
systems. Clustering typically means that certain related events occur
``in proximity to each other'' and, when it happens, the impact of the
events may be magnified, as in the situations we started
with. Clustering is also interesting in its own right because it may shed
light on certain structural elements in a random system. Clustering is
most frequently studied in distributional analysis; an important
example is clustering of extreme values; see
e.g. \cite{embrechts:kluppelberg:mikosch:2003}. In the above examples,
however, one is interested in clustering of ``unexpected'' values over
long time periods. The point of view of large deviations is,
therefore, called for.

In this work we are discuss clustering of large deviations
events. From a different point of view, we would like to understand
whether or not a (rare) large deviations event is likely to cause a
cascade of additional large deviations events and, if so, what does
this cascade look like. Literature on large deviations analysis is
vast, however to the best of our knowledge, clustering has not been
considered even in the i.i.d. case; the applications we have in mind
call for models with dependence. The nature of large deviations is
known to be different in
stochastic systems with ``light tails'' and with ``heavy tails''. The
texts such as \cite{dembo:zeitouni:1998} or
\cite{deuschel:stroock:1989} describe large deviations of light-tailed
systems, while \cite{mikosch:nagaev:1998} will give the reader an idea
how large deviations occur in heavy-tailed systems. Large deviations
are affected not only by the ``tails'' in a random system, but also
the ``memory'' in that system, in particular by whether the memory is
``short'' or ``long''. The change from short to  long memory in a 
system can be viewed as a phase transition (see
\cite{samorodnitsky:2016}), and it affects large deviations as
well. In this work we study clustering of large deviations in a
light-tailed system, but we will consider both short memory and long
memory situations.

Let us now be more specific about the class of stochastic models we
will consider. We will consider centred infinite moving average processes
\begin{equation} \label{e:MA}
X_n=\sum_{i=-\infty}^\infty a_i Z_{n-i}\,,\ n\in\bbz\,,
\end{equation}
where $(Z_n:n\in\bbz)$ is a collection of i.i.d.\ non-degenerate
random variables ({\it the noise}) with distribution $F_Z$ satisfying 
\begin{equation}  \label{eq.exp}
  \int_\bbr e^{tz}\, F_Z(dz)<\infty\text{ for all }t\in\bbr\,,
\end{equation}
and
\begin{equation} \label{ima.eq1}
  \int_\bbr z\, F_Z(dz) =0\,.
\end{equation}
For future reference we denote
\begin{equation}
  \label{eq.defsigmaz}
  \sigma_Z^2=\int_\bbr z^2\, F_Z(dz). 
\end{equation}

Let  $\ldots, a_{-1}, a_0,a_1,a_2\ldots$ be real numbers satisfying 
\begin{equation}   \label{eq.sqsum}
  \sum_{j=-\infty}^\infty a_j^2<\infty\,.
\end{equation}
Since the assumption \eqref{eq.exp} implies that the noise variables
have a finite second moment, the zero mean property assumed in \eqref{ima.eq1} and the square integrability of the coefficients \eqref{eq.sqsum} imply that  the infinite sum in the
right hand side  of \eqref{e:MA} 
converges in $L^2$ and a.s.\  and defines a zero mean stationary ergodic 
process. Therefore, for $\vep>0$ the event
$$
E_0(n,\vep)=\left\{\frac1{n}\sum_{i=0}^{n-1}X_i\ge \vep\right\}
$$
is, for large $n$, a rare, large deviations, event. We would like to
understand whether occurrence of this event may cause a cascade of
related events. Specifically, for $j\geq 0$ we denote 
\begin{equation} \label{e:E.j}
E_j(n,\vep)=\left\{\frac1{n}\sum_{i=j}^{n+j-1}X_i\ge \vep\right\}\,,
\end{equation}
so that each event $E_j(n,\vep)$ is equally rare, and we would like to
know how many of the events for $j$ ``reasonably close to $j=0$'' 
occur if $E_0(n,\vep)$ occurs (the reason for the qualifier
``reasonably close to $j=0$'' is that by ergodicity, the events
$E_j(n,\vep)$ will keep recurring eventually, regardless of the
structure of the system).

The difference between short memory infinite moving average processes
and  long memory infinite moving average processes lies in the rate
the coefficients $(a_n)$ converge to zero (subject to the square
summability, of course). This will lead to markedly different
cascading of the events $E_j(n,\vep)$, conditionally on the event
$E_0(n,\vep)$ occurring. In this paper we consider the short  
memory processes; the long memory case is studied in another
paper -- \cite{chakrabarty:samorodnitsky:2023}.  Specifically, we describe the limiting distribution of the
large deviation cluster caused by the rare event  $E_0(n,\vep)$ as
well as the behaviour of the size of that cluster as the overshoot
$\vep$ becomes small. It turns out that for such $\vep$ the size of
the cluster is of the order $\vep^{-2}$.

Our main results on the cluster of large deviations for   short memory
infinite moving average processes are  in
Section \ref{sec:short.mem}. They are presented in a somewhat more
general way than discussed so far. Even though the ``causal'' point of
view encourages us to consider the ``forward'' large deviation events 
$E_j(n,\vep)$ in \eqref{e:E.j} for $j\geq 0$, in certain situations it
is also interesting to look at the ``backward'' large deviation events 
$E_j(n,\vep)$ with $j<0$, and this is included in Section \ref{sec:short.mem}. 
The concluding Section
\ref{sec:discussion} contains a discussion of the results we obtain
and connects them to what one may expect when the memory becomes
long. 

\section{Short memory moving average processes} \label{sec:short.mem}

We follow the common terminology and say that the infinite moving
average process \eqref{e:MA} has short memory if
\begin{equation} \label{e:sm.def}
\sum_{n=-\infty}^\infty |a_n|<\infty \ \ \text{and} \ \ \sum_{n=-\infty}^\infty
a_n\not= 0.
\end{equation}
Large and moderate deviations for such processes have been studied in,
for example, \cite{jiang:wang:rao:1992} and 
\cite{djellout:guillin:2001}. 

We investigate the clustering of the rare events $(E_j(n,\vep))$ in
the following way. We will show that the conditional law of the
process of occurrences of the large deviation events, 
\begin{equation} \label{e:exceed.process}
  \bigl( \one(  E_j(n,\vep), \, j\in\bbz)\bigr), 
\end{equation}
given $E_0(n,\vep)$ has a non-degenerate weak limit and describe that
limit. This will show, in particular, that for  fixed $K_-,\, K_+\in\bbz$
the joint 
conditional law of the total numbers of occurrences among the last
$K_-$ of the events $(E_j(n,\vep))$ before 0 and the 
first $K_+$
of the events $(E_j(n,\vep))$ after 0, 
\begin{equation} \label{e:law.K}
\nu_n(K_-,K_+,\vep)(\cdot)=
P\left[\left(\sum_{j=-K_-}^{-1}\one(E_j(n,\vep)),\, \sum_{j=1}^{K_+}\one(E_j(n,\vep))\right)\in\cdot\Biggr|E_0(n,\vep)\right]
\end{equation}
has a weak limit.  That weak limit itself converges weakly, as
$K_-, \, K_+\to\infty$, to an a.s.\  finite random variable that we interpret as
the size of the cluster of large deviation events containing a large
deviation event at time zero. An interesting regime is that of a small
$\vep>0$, and we show that a properly normalized size of the cluster of
large deviation events converges weakly, as $\vep\to 0$, 
to an
(interesting) limit.  As we have explained above, the limits 
should be taken in this specific order.

To state the main results of this section we need to introduce some
notation first. Denote 
\begin{equation}
  \label{ima.eq3}
A=\sum_{n=-\infty}^\infty a_n = \sum_{n=-\infty}^{-1} a_n + \sum_{n=0}^\infty a_n := A^-+A^+. 
\end{equation}
In the sequel we will assume that $A>0$. Note that, in view of
\eqref{e:sm.def}, this introduces no real loss of generality because,
if the sum is negative, we simply multiply both  $(Z_n)$ and $(a_n)$
by $-1$ and reduce the situation to the case $A>0$ we are
considering. Further, for $n=0,1,2,\ldots$ we write 
\begin{equation}
\label{sm.eq0} A_n^-= \sum_{j=1}^n a_{-j}, \ 
A_n^+=\sum_{j=0}^na_j \,.
\end{equation}

Next, we let 
\begin{equation}
\label{ima.eq.varphiz}\varphi_Z(t)=\log\left(\int_\bbr e^{tz}\,
  F_Z(dz)\right),\ t\in\bbr
\end{equation}
to be the log-Laplace transform of a noise variable. 
For $\theta\in\bbr$ we denote by $G_\theta$ the probability measure on
$\bbr$  obtained by exponentially tilting $F_Z$ as follows: 
\begin{equation}
\label{sm.eqdefg}G_\theta(dx)=e^{-\varphi_Z(\theta)+\theta x}F_Z(dx),\
x\in\bbr\,.
\end{equation}
Further, let 
\[
s_0=\sup\{x\in\bbr:F_Z(x)<1\}\in (0,\infty]
\]
be the right endpoint of the support of a noise variable.  We will
consider the events $(E_j(n,\vep))$ for $\vep$ satisfying 
\begin{equation}   \label{sm.eqvep}
  0<\vep/A<s_0\,.
\end{equation}
The function $\varphi_Z$ is infinitely differentiable, and its first
derivative $\varphi_Z^\prime$ strictly increases from $0$ at $t=0$ to
$s_0$ as $t\to\infty$. Therefore, for $\vep$ satisfying
\eqref{sm.eqvep}, we can unambiguously define $\tau(\vep)>0$ by 
\begin{equation}   \label{sm.eqtau}
  \varphi_Z^\prime\left(\tau(\vep)\right)= \vep/A\,.
\end{equation}
 
We now introduce a collection  $\{Z_j^u:j\in\bbz, \, u=+\text{ or
}-\}$ of independent random variables with the following laws:
\begin{equation} \label{e:Z.shifted}
\begin{array}{ll}
Z^-_{-j}\sim G_{\tau(\vep)(A^+-A^+_{j-1})/A}\,,&j\ge1\,,\\
Z^-_j\sim G_{\tau(\vep)(A^++A^-_j)/A}\,,&j\ge0\,,\\
Z^+_{-j}\sim G_{\tau(\vep)(A_{j-1}^++A^-)/A}\,,&j\ge1\,,\\
Z^+_{j}\sim G_{\tau(\vep)(A^--A_j^-)/A}\,,&j\ge0\,. 
\end{array}
\end{equation}
Finally,  let $T^*$ be an exponential random variable with parameter
$\tau(\vep)/A$, independent of the family \eqref{e:Z.shifted}. 

It is elementary to check that for any $j\in \bbz$ and $u=+\text{ or
}-$, 
\begin{equation*}
  E\left(|Z_j^u|\right)\le\int_\bbr |x|\, G_{A^{-1}\bar
    A\tau(\vep)}(dx)\,+\int_\bbr |x|\, G_{-A^{-1}\bar
    A\tau(\vep)}(dx)<\infty\,,
\end{equation*}
where
$$
\bar A=\sum_{n=-\infty}^\infty |a_n|.
$$
Therefore, the infinite series 
\begin{align} \label{e:U}
U^-_n&=\sum_{i=-\infty}^\infty a_iZ^-_{n-i}, \ n\in\bbz\,,\\
\notag U^+_n&=\sum_{i=-\infty}^\infty a_iZ^+_{n-i}, \ n\in\bbz 
\end{align}
converge in $L^1$ and a.s.  We define
\begin{equation} \label{e:V}
V_j(\vep)=\left\{ \begin{array} {ll}
                    \one\left(T^*\ge\sum_{i=0}^{j-1}(U_i^--U_i^+)\right),
                    & j\ge1, \\
                 \one\left(T^*\ge\sum_{i=j}^{-1} (U_i^+ -U_i^-)\right), & j<0.
\end{array} \right.
\end{equation}

\begin{remark} \label{rk:iid}
  {\rm
It is instructive to see what the above random objects become in the
i.i.d. case $a_0=1, \, a_i=0$ for all $i\not=0$. Then $U^-_{-j}=Z^-_{-j}\sim G,
\, j\ge1, \, U^-_{j}=Z^-_j\sim G_{\tau(\vep)}, \, j\geq 0, \, U^+_{-j}=Z^+_{-j}\sim
G_{\tau(\vep)}, \, j\geq 1, \, U^+_{j}=Z^+_{j}\sim G, \, j\geq 0$. The
processes $\sum_{i=0}^{j-1}(U_i^--U_i^+), \, j\geq 1$ and
$\sum_{i=j}^{-1} (U_i^+ -U_i^-), \, j<0$ become independent random
walks with the same step distribution, that of the difference  $A-B$
of independent random variables, $A\sim G_{\tau(\vep)}$, $B\sim
G$. That is, $V_j(\vep)$ for both positive and negative $j$ are equal
to 1 until the first times these random walks cross level $T_*$, after
which they vanish. 
  }
\end{remark}

We are now ready to state the main theorems of this section. They rely
on a technical assumption, excluding the case of a lattice-valued
noise. We assume that 
\begin{equation} \label{ima.eq0}
  \left|\int_\bbr e^{i tz} \,
    F_Z(dz)\right|<1\ \ \text{for any $t\not=0$ where $i=\sqrt{-1}$.}
\end{equation}
Our first result describes the behaviour 
of the sequence of
conditional laws of the process of the overshoots of the  level
$n\vep$ by the partial sums of length $n$. This leads to the limiting
behaviour of  the sequence of
conditional laws of the process of 
occurrences of the large deviations
events \eqref{e:exceed.process} and of the sequence 
\eqref{e:law.K}  of the total number of occurrences
among the first $K$ of the events $(E_j(n,\vep))$. The weak
convergence of sequences of random variables stated in this theorem
occurs in the usual topology of finite-dimensional convergence in $\bbr^\infty$ (or its restriction to
$\{0,1\}^\infty$). 
\begin{theorem}\label{sm.t1}
Assume that \eqref{e:sm.def} holds and $A>0$ in
\eqref{ima.eq3}. Assume, further, that the characteristic function of
the noise variables satisfies \eqref{ima.eq0}. Let  $\vep$ be as in
\eqref{sm.eqvep}. Then, as $n\to\infty$, 
\begin{align} \label{e:overshoots}
&P\left( \left( \sum_{i=j}^{j+n-1} X_i-n\vep, \, j\in\bbz\right)
  \in\cdot \Bigg| E_0(n,\vep)\right) \\
\notag \Rightarrow
&\Biggl( 
T^*+\sum_{i=j}^{-1} U_i^--\sum_{i=j}^{-1} U_i^+,\   j=\ldots,-2, -1,\\
\notag &
 \hskip 0.1in T^*-\sum_{i=0}^{j-1}U_i^- +
                                          \sum_{i=0}^{j-1} U_i^+, \ j=0,1,2,\ldots\Biggr)   
\end{align}
in $\bbr^\infty$. In particular, 
\begin{equation} \label{e:seq.conv}
P\Bigl(  \bigl( \one(  E_j(n,\vep), \,
j\in\bbz)\bigr)\in\cdot \Big| E_0(n,\vep)\Bigr) \Rightarrow P\Bigl(\bigl( 
V_j(\vep), \, j\in\bbz\bigr)\in\cdot\Bigr)  
\end{equation}
in $\{0,1\}^\infty$.  Furthermore,  for 
 every fixed $K_-,\, K_+\ge1$, the conditional laws $(\nu_n(K_-,K_+,\vep))$ in
\eqref{e:law.K} satisfy 
\begin{align}   \label{sm.t1.claim}
 \nu_n(K_-,K_+,\vep)(\cdot) 
\Rightarrow P\left[\left(\sum_{j=-K_-}^{-1}V_j(\vep), \, \sum_{j=1}^{K_+}V_j(\vep)\right)\in\cdot\right]\ \ \text{as} \
  n\to\infty\,.
\end{align}
\end{theorem}

\begin{remark} \label{rk:EVT}
{\rm 
The limiting process of overshoots appearing in the right hand side
of \eqref{e:overshoots} has a remarkable property. Define
\begin{equation} \label{e:over.Y}
  Y_j = \left\{ \begin{array}{ll}
  \exp\left\{ T^*+\sum_{i=j}^{-1} U_i^--\sum_{i=j}^{-1} U_i^+\right\},
                  & j=\ldots,-2, -1,                \\
\exp\left\{ T^*-\sum_{i=0}^{j-1}U_i^- +
                                          \sum_{i=0}^{j-1} U_i^+\right\}, & j=0,1,2,\ldots.
\end{array} \right.
\end{equation}
We claim that the process $\BY=\bigl( Y_j, \, j\in\bbz\bigr)$
satisfies the following time change formula:
\begin{equation} \label{e:time.change}
E\bigl[ H\bigl( tB^j\BY\bigr) \one(Y_{-j}>1/t)\bigr] =
t^{\tau(\vep)/A} E\bigl[ H(\BY) \one(Y_{j}>t)\bigr] 
\end{equation}
for every $j\in\bbz, \, t>0$ and every measurable bounded functional
$H$ on $[0,\infty)^\infty$. Here $B$ is the usual backward
shift operator on $[0,\infty)^\infty$.

The time change property \eqref{e:time.change} has been known to be
a important property of the so-called {\it tail process} in extreme
value theory. Specifically, if $\bigl( W_j, \, j\in\bbz\bigr)$ is a
nonnegative stationary process with multivariate regularly varying
tails with exponent $\alpha>0$, then the conditional law of $\bigl( W_j/W_0, \,
j\in\bbz\bigr)$ given $W_0>w$ has a weak limit, as $w\to\infty$. This
limit is the law of the tail process, described by
\cite{basrak:segers:2009}, who also discovered  the time change
property, with the exponent $\tau(\vep)/A$ replaced by $\alpha$ (allowing non-necessarily nonnegative values, and
in dimensions greater than 1).
This property has become  subject of further investigations; we
refer the reader to its recent presentation in a book form in Theorem 5.3.1 in \cite{kulik:soulier:2020}. 

The tail process in extreme value theory arises in the distributional
context of tail behaviour of stationary processes with regularly
varying tails. The limiting process $\BY$ in \eqref{e:over.Y} arises
in the context of large deviations of stationary moving average
processes, without any presence of regular variation in the underlying
model. This makes
appearance of the time change property here unexpected. It
would be interesting to investigate how widespread this phenomenon is
in the realm of large deviations, but we are not pursuing this
question any further in this paper.

In order to prove that the process $\BY$ in \eqref{e:over.Y} satisfies
\eqref{e:time.change}, suppose first that $j\geq 0$. It is elementary
that the Pareto-distributed random variable $W=e^{T*}$ has the property
\begin{equation} \label{e:pareto.prop}
t^{\tau(\vep)/A}E\bigl[ h(W)\one(W>t)\bigr] = E\bigl[
h(tW)1(W>1/t)\bigr],
\end{equation}
valid for any $t>0$ and any bounded measurable function $h$. Using
this property we have
\begin{align*}
 & t^{\tau(\vep)/A} E\bigl[ H(\BY) \one(Y_{j}>t)\bigr] \\
&=
t^{\tau(\vep)/A} E\Biggl[ H\Biggl( W\exp\left\{ \sum_{i=k_1}^{-1}
  (U_i^--U_i^+)\right\}, \ k_1=\ldots, -2, -1, \\
& \hskip 1.1in W \exp\left\{ \sum_{i=0}^{k_2-1} (U_i^+-U_i^-)\right\}, \ k_2=0,1,\ldots\Biggr)\\
&\hskip 1in \one\left( W\exp\left\{ \sum_{i=0}^{j-1}  (U_i^+-U_i^-)\right\}>t\right)\Biggr] \\
&= E\Biggl[ H\Biggl( tW\exp\left\{ \sum_{i=k_1}^{-1}
  (U_i^--U_i^+)\right\}, \ k_1=\ldots, -2, -1, \\
&  \hskip 0.67in 
tW \exp\left\{ \sum_{i=0}^{k_2-1} (U_i^+-U_i^-)\right\}, \ k_2=0,1,\ldots\Biggr)\\
&\hskip 0.5in \exp\left\{ \frac{\tau(\vep)}{A}\sum_{i=0}^{j-1}
          (U_i^+-U_i^-)\right\}
\one\left( W\exp\left\{ \sum_{i=0}^{j-1}  (U_i^--U_i^+)\right\}>1/t\right)\Biggr]. 
\end{align*}
The factor
$$
\exp\left\{ \frac{\tau(\vep)}{A}\sum_{i=0}^{j-1}
  (U_i^+-U_i^-)\right\}
$$
under the expectation in the right hand side introduces additional
exponential tilting into the random variables $Z_d^{\pm}$ in
\eqref{e:Z.shifted}. Specifically, it is a matter of straightforward
calculus to check that, distributionally, $Z_d^-$ becomes $Z_{d-j}^-$
and $Z_d^+$ becomes $Z_{d-j}^+$ for every $d\in\bbz$. That is,
distributionally, $U_n^{\pm}$ becomes $U_{n-j}^{\pm}$ for every
$n\in\bbz$. Therefore,
\begin{align*}
 & t^{\tau(\vep)/A} E\bigl[ H(\BY) \one(Y_{j}>t)\bigr] \\
&= E\Biggl[ H\Biggl( tW\exp\left\{ \sum_{i=k_1}^{-1}
  (U_{i-j}^--U_{i-j}^+)\right\}, \ k_1=\ldots, -2, -1, \\
&  \hskip 0.67in 
tW \exp\left\{ \sum_{i=0}^{k_2-1} (U_{i-j}^+-U_{i-j}^-)\right\}, \ k_2=0,1,\ldots\Biggr)\\
&\hskip 0.5in  
\one\left( W\exp\left\{ \sum_{i=0}^{j-1}  (U_{i-j}^--U_{i-j}^+)\right\}>1/t\right)\Biggr]\\
&= E\bigl[ H\bigl( tB^j\BY\bigr) \one(Y_{-j}>1/t)\bigr], 
\end{align*}
with the last step following from a rearrangement of indices. This
proves \eqref{e:time.change} for $j\geq 0$. Finally, if $j<0$, then we
apply \eqref{e:time.change} to $\tilde j=-j$, $\tilde t = 1/t$ and
$\tilde H(\cdot) = H(tB^j\cdot)$, proving \eqref{e:time.change}  in
this case as well. 
}
\end{remark}

\begin{remark} \label{rk:HT}
{\rm 
The conditional limiting behaviour described in Theorem \ref{e:sm.def}
reflects the nature of large deviations for weighted sums of random
variables with exponentially light tails: the terms in the sums
``conspire  to change their distributions just right'' to make the
rare event happen (and the change in distributions is reflected in the
exponential tilting). As the tails become heavier, the nature of
large deviations gradually changes from a ``conspiration'' to a
``single extraordinary value'' phenomenon. This will result in a
change of how large deviation events cluster. We outline it in  a
simple example and leave a full discussion to a different occasion. 
Consider the finite moving average $X_n=Z_n+Z_{n-1}, \, n\in\bbz$,
where $(Z_n)$ are i.i.d. 0 mean random variables whose right tail is
regularly varying with exponent $\alpha>1$. Then
$$
S_n = Z_{-1}+Z_{n-1} + 2\sum_{i=0}^{n-2} Z_i,
$$
and the event $E_0(n,\vep)$ is, asymptotically, equivalent to the
event that (exactly) one of the $Z_i$ with $i=0,1,\ldots, n-2$ is
larger than $n\vep/2$, with equal probabilities (it also possible that
$Z_1$ or $Z_{n-1}>n\vep$, but the probability of this goes to 0 as
$n\to\infty$.)  For any $j\geq 1$, the probability that this
exceptional $i$ is in the range $j,\ldots, n-2$ converges to 1, and
the corresponding term $2Z_i$ is also a part of the sum $\sum_{j=i}^{j+n-1}
Z_j$. Therefore, the conditional probability of $E_j(n,\vep)$ given
that $E_0(n,\vep)$ occurs converges to 1 for every $j\geq 1$, as
$n\to\infty$. 
}
\end{remark}

It is natural to interpret the statement of Theorem \ref{sm.t1} as
saying that a large deviation event $E_0(n,\vep)$, upon occurring,
leads to random clusters of large deviation events in the fast and in
the future. The limiting (as
$n\to\infty$)  total sizes of these clusters have  the joint law of
\begin{equation} \label{e:D}
(D_\vep^-,D_\vep^+)=\left( \sum_{j=-\infty}^{-1}
  V_j(\vep),\, \sum_{j=1}^\infty V_j(\vep)\right), 
 \ \vep>0\,. 
\end{equation} 
Our second result of this section shows that for small $\vep>0$  these total cluster sizes
are a.s.\  finite and 
describes their joint limiting behaviour as the overshoot $\vep\to 0$.

\begin{theorem}\label{sm.t2} Under the assumptions of Theorem
  \ref{sm.t1}, the total cluster sizes $D^-_\vep ,D_\vep^+$ are 
  a.s.\  finite for $\vep>0$ small enough. Further,  as $\vep\to0$,
  \begin{align*}
    (\vep^2D_\vep^-,\vep^2D_\vep^+)
 \Rightarrow \Biggl( &A^2\sigma_Z^2\int_0^\infty
\one\left(T_0\ge(\sqrt2B^-_t+t)\right) dt, \\
 &A^2\sigma_Z^2\int_0^\infty
\one\left(T_0\ge(\sqrt2B^+_t+t)\right) dt\Biggr), 
\end{align*}
where $A$ is the sum of the coefficients \eqref{ima.eq3} and
$\sigma_Z^2$ is the noise 
variance \eqref{eq.defsigmaz}. Furthermore, 
$T_0$ is a standard exponential random variable independent of two
independent 
standard Brownian motions $(B_t^\pm:t\ge0)$. 
\end{theorem}

\bigskip

We prove Theorem \ref{sm.t1} first, and so $\vep>0$ (satisfying
\eqref{sm.eqvep}) is for now fixed. The proof is via several technical
lemmas, and we first sketch the flow of the argument. 
To simplify the notation we will write $E_j$ instead of
$E_j(n,\vep)$ throughout. We start by
deriving a non-logarithmic asymptotic formula for the probability of
$E_0$, which we use to show that, conditionally on $E_0$, all noise
variables remain uniformly bounded in $L_1$ and, further, jointly
weakly converge to the appropriately exponentially tilted laws. This
allows us to prove that the sums of finitely truncated moving averages
converge weakly, and this takes us very close to the finish.

We now embark on the technical details. 
\begin{lemma}\label{sm.l1}
Denote 
\begin{equation}
\label{eq.defsn}S_n=\sum_{i=0}^{n-1}X_i, \ n\ge1\,,
\end{equation} 
and let 
\begin{equation} \label{e:psi}
\psi_n(t)=n^{-1}\log E\left(e^{tS_n}\right), \ t\in\bbr,\ n\ge1\,.
\end{equation}
Then for all large enough $n$  there exists a unique $\theta_n>0$ such that
\[
\psi_n^\prime(\theta_n)=\vep\,.
\]
Furthermore,
\[
P(E_0)\sim\frac C{\sqrt
  n}\exp\Bigl(-n\bigl(\theta_n\vep-\psi_n(\theta_n)\bigr)\Bigr), \ n\to\infty\,,
\]
with 
\[
C=\frac1{\tau(\vep)\sqrt{2\pi\varphi_Z^{\prime\prime}(\tau(\vep))}}\,,
\]
and $\varphi_Z$ and $\tau(\vep)$ defined, respectively,  in
\eqref{ima.eq.varphiz} and \eqref{sm.eqtau}. 
\end{lemma}
\begin{proof}
We write 
\begin{align} \label{sm.l1.eq-1}
S_n 
&=\sum_{j=0}^{n-1}(A^+_{n-1-j}+A^-_j)Z_j+\sum_{j=n}^\infty(A^-_j-A^-_{j-n}) Z_j\\
\nonumber&\,\,\,\,\, +\sum_{j=1}^\infty(A^+_{j+n-1}-A^+_{j-1})Z_{-j}\,,
\end{align}
with $A^+_n,A^-_n$ defined by \eqref{sm.eq0} and check the conditions of
Theorem \ref{f.cs} in the appendix. As a first step we show that
\begin{equation}
\label{sm.l1.eq0}\lim_{n\to\infty}\psi_n^\pp(t)=A^2\varphi_Z^\pp(At)
\end{equation}
locally uniformly in $t\in\bbr$. Indeed, by \eqref{sm.l1.eq-1}, 
\begin{align} \label{e:psi.n}
\psi_n(t)=&\frac1n\Biggl[\sum_{j=0}^{n-1}\varphi_Z\bigl(
  (A^+_{n-1-j}+A^-_j)t\bigr) +\sum_{j=n}^\infty
            \varphi_Z\bigl((A^-_j-A^-_{j-n})t\bigr) \\
\notag &\hskip 0.1in + \sum_{j=1}^\infty\varphi_Z\left((A^+_{j+n-1}-A^+_{j-1})t\right)\Biggr].
\end{align}
It is easy to see that
\begin{equation} \label{e:first.t}
\lim_{n\to\infty}\frac1n\sum_{j=0}^{n-1}  (A^+_{n-1-j}+A^-_j)^2
\varphi_Z^\pp \bigl(
  (A^+_{n-1-j}+A^-_j)t\bigr)=A^2\varphi_Z^\pp(At)\,.
\end{equation}
Indeed, for every   $0<\eta<1/2$, 
$$
\frac1n\sum_{j=[\eta n]}^{[(1-\eta)n]}  (A^+_{n-1-j}+A^=_j)^2
\varphi_Z^\pp \bigl(
  (A^+_{n-1-j}+A^-_j)t\bigr) \to (1-2\eta) A^2\varphi_Z^\pp(At),
$$
since the terms in the sum converge uniformly to the limit. 
Furthermore, since $\varphi_Z^\pp$ is locally bounded,
$$
\frac1n\sum_{j=0}^{[\eta n]-1}  (A^+_{n-1-j}+A^-_j)^2
\varphi_Z^\pp \bigl(
(A^+_{n-1-j}+A^-_j)t\bigr) \leq C\eta
$$
for some finite constant $C$. Since the second remaining part of the
sum in \eqref{e:first.t} can be bounded in the same way, by letting
first $n\to\infty$ and then $\eta\to 0$ we obtain \eqref{e:first.t}. 
Furthermore, as $n\to\infty$, 
\begin{align*}
\sum_{j=1}^\infty(A^+_{j+n-1}-A^+_{j-1})^2\varphi_Z^\pp\left((A^+_{j+n-1}-A^+_{j-1})t\right)&=O\left(\sum_{j=1}^\infty(A_{j+n-1}^+-A^+_{j-1})^2\right)\\
&=O\left(\sum_{j=1}^\infty\sum_{i=j}^{j+n-1}|a_i|\right)\\
&=O\left(\sum_{i=1}^{n}\sum_{j=i}^\infty|a_j|\right) =o(n). 
\end{align*}
Since the second sum in \eqref{e:psi.n} can be estimated in the same
way, and all these steps are locally uniform in $t\in\bbr$, 
\eqref{sm.l1.eq0} follows. Since this argument also shows that
$\psi_n^\pp$ is, uniformly in $n$, locally bounded, and the  values of
 $\psi_n, \varphi_Z$ and their respective first derivatives at $0$ are
 $0$, we also conclude that $\psi_n^\prime$ is, uniformly in $n$,
 locally bounded and  for every $t\geq 0$, 
\begin{align}
\label{sm.l1.eq1}\lim_{n\to\infty}\psi_n^\prime(t)&=A\varphi_Z^\prime(At)\,,\\
\notag \lim_{n\to\infty}\psi_n(t)&=\varphi_Z(At)\,.
\end{align}
The assumption \eqref{sm.eqvep} together with \eqref{sm.l1.eq1}
implies that for large $n$  there exists a unique $\theta_n>0$ such that $\psi_n^\prime(\theta_n)=\vep$, and that
\begin{equation}
\label{sm.l1.eq3}\lim_{n\to\infty}\theta_n=A^{-1}\tau(\vep)\,.
\end{equation}
We choose $n_1$ so large that for $n\geq n_1$, $\theta_n$ is well defined, 
$A/2\le A_{n-1-j}+B_j\le\sqrt2A$ for all $n/3\leq j\leq 2n/3$ 
and $\theta_n\le\sqrt2\tau(\vep)/A$.

We claim next that for fixed $\delta,\lambda>0$ there exists $\eta\in(0,1)$ such that
\begin{equation}
\label{sm.l1.eq4}\sup_{\delta\le|t|\le\lambda\theta_n}\left|\frac1{E(e^{\theta_nS_n})}E\left(e^{(\theta_n+it)S_n}\right)\right|=O\left(\eta^n\right),
\ n\to\infty,
\end{equation}
with the convention that the supremum of the empty set is zero. To see this, note that $\phi:\bbr^2\to\bbc$ defined by
\[
\phi(\theta,t)=\frac1{ E(e^{\theta Z})} E\left(e^{(\theta+i t)Z}\right)\,,
\]
is continuous. For a fixed $\theta\in\bbr$, $\phi(\theta,\cdot)$ is
the characteristic function of the distribution $G_\theta$  in
\eqref{sm.eqdefg}. By \eqref{ima.eq0}, $G_\theta$ is not a lattice
distribution and, hence, for any fixed $\lambda,\delta>0$ and $\theta$,
\[
\sup_{A\delta/2\le|t|\le2\lambda\tau(\vep)}\left|\phi(\theta,t)\right|<1\,.
\]
A standard compactness argument and \eqref{sm.l1.eq3} imply that 
\[
\eta:=\left(\sup_{n\ge n_1, \, n/3\leq j\leq 2n/3,\, 
  A\delta/2\le|t|\le2\lambda\tau(\vep)}\left|\phi\bigl(\theta_n(A^+_{n-1-j}+A^-_j),t\bigr)\right|\right)^{1/3}<1\,, 
\]
while the choice of $n_1$ implies that for 
$n>j\ge n_1$ and $\delta\le|t|\le\lambda\theta_n$,  
\[
A\delta/2\le|(A^+_{n-1-j}+A^-_j)t|\le2\lambda\tau(\vep)\,.
\]
Therefore, by  \eqref{sm.l1.eq-1} and the triangle inequality, 
\begin{align*}
\frac1{ E(e^{\theta_nS_n})}\left| E(e^{(\theta_n+i
  t)S_n})\right|&\le\prod_{j=[n/3]+1}^{\lceil 2n/3\rceil}
                  \left|\phi\bigl(\theta_n(A^+_{n-1-j}+A^-_j),t(A^+_{n-1-j}+A^-_j)\bigr)\right|\\
&\le\eta^{n}\,,
\end{align*}
establishing \eqref{sm.l1.eq4}.

We have now verified all conditions of Theorem \ref{f.cs} for
$T_n=S_n$,  $a_n=n$,  $m_n=\vep$ and $\tau_n=\theta_n$, and
\eqref{cs.eq6} gives us the statement of the lemma. 
\end{proof}

We proceed with showing uniform boundedness of conditional moments of
all noise variables. 
\begin{lemma}\label{sm.l2}
We have 
\begin{equation} \label{e:inf.mom}
\sup_{n\geq 1,\, j\in\bbz} E\left(|Z_j|\bigr| E_0\right)<\infty\,.
\end{equation}
\end{lemma}
\begin{proof}
Fix $j\in\bbz$ and  define 
\[
S_{n,j}=S_n-\beta_{n,j}Z_j, \ n\geq 1, 
\]
where
\[
  \beta_{n,j}=\begin{cases}
    A^-_j-A^-_{j-n}, &    1\leq n\leq j \\
A^+_{n-1-j}+A^-_j,& n\geq j+1
\end{cases}
\]
if $j\geq 0$ and
\[
  \beta_{n,j}= A^+_{n-1-j}-A^+_{-j-1}, \ n\geq 1
\]
if $j\leq -1$, with  $S_n$ is as in \eqref{eq.defsn}.  It follows from
\eqref{sm.l1.eq-1}  that $S_{n,j}$ and $Z_j$ are independent. 
We define
\begin{equation} \label{e:psi.tilde}
\tilde\psi_{n,j}(t)=n^{-1} \log E\left(e^{tS_{n,j}}\right), \ t\in\bbr\,.
\end{equation} 
Since the numbers $(\beta_{n,j})$ are bounded uniformly in $j$ and
$n$, it follows that the functions in \eqref{e:psi} and
\eqref{e:psi.tilde} satisfy 
\[
\psi_{n}^\prime(\theta)=\tilde\psi_{n,j}^\prime(\theta)+O(1/n)\,,
\]
with $O(1/n)$ uniform over $j$  and $\theta$ in a compact
interval. The same argument as in Lemma \ref{sm.l1} shows that for
large $n$ there exists a unique $\tilde\theta_{n,j}>0$ such that
\[
\tilde\psi_{n,j}^\prime(\tilde\theta_{n,j})=\vep\,.
\]
Since $\varphi_Z^\prime$ is locally bounded away from zero,  
it follows from \eqref{sm.l1.eq1} that 
\begin{equation}
\label{sm.l2.eq1}\tilde\theta_{n,j}=\theta_n+O(1/n) 
\end{equation}
with $O(1/n)$ uniform over $j$. This also implies that 
\begin{equation}
\label{sm.l2.eq2}\psi_n(\tilde\theta_{n,j})=\psi_n(\theta_n)+O(1/n)\,.
\end{equation}

For large $n$ we can write 
\begin{align*}
 E\left(|Z_j|\one(E_0)\right)
&=\int_\bbr |z|P(Z_j\in dz)\int_\bbr \one(s+\beta_{n,j} z\ge n\vep) P(S_{n,j}\in ds)\\
  &=\exp\left\{-n\left(\tilde\theta_{n,j}\vep-\psi_n(\tilde\theta_{n,j})\right)\right\}
    \int_\bbr |z|e^{-\tilde\theta_{n,j}\beta_{n,j}z} P(Z_{n,j}^*\in dz)\\
&\,\,\,\,\,\int_{[n\vep-\beta_{n,j}z,\infty)}\exp\left(-\tilde\theta_{n,j}(s-n\vep)\right)
    P(S_{n,j}^*\in     ds)\,,
\end{align*}
where $S_{n,j}^*$ and $Z_{n,j}^*$ are independent random variables with
$Z_{n,j}^*$ having distribution $G_{\tilde\theta_{n,j}\beta_{n,j}}$ and 
\[
P(S_{n,j}^*\in ds)=\frac1{ E\left(\exp(\tilde\theta_{n,j}S_{n,j})\right)}e^{\tilde\theta_{n,j}s}P(S_{n,j}\in ds)\,.
\]
It follows from \eqref{sm.l2.eq1} and \eqref{sm.l2.eq2} that,
uniformly over $j$, 
\begin{align*}
\exp\left\{-n\left(\tilde\theta_{n,j}\vep-\psi_n(\tilde\theta_{n,j})\right)\right\}&=O\left(\exp\left(-n\bigl(\theta_n\vep-\psi_n(\theta_n)\bigr)\right)\right)\\
&=O\left(\sqrt nP(E_0)\right)\,,
\end{align*}
with the second line implied by Lemma \ref{sm.l1}. Therefore, to
complete the proof it suffices to show that, uniformly in $j$, 
\begin{align} \label{sm.l2.eq3}
&\int_\bbr |z|e^{-\tilde\theta_{n,j}\beta_{n,j}z} P(Z_{n,j}^*\in
   dz)\int_{[n\vep-\beta_{n,j}z,\infty)}\exp\left(-\tilde\theta_{n,j}(s-n\vep)\right)
   P(S_{n,j}^*\in ds) \\
\notag &=O\left(n^{-1/2}\right)\,.
\end{align}
This will follow from the following claim: there is $C^\prime>0$ such
that for all $n$ large, 
\begin{equation} \label{e:claim.un}
P(y\le S_{n,j}^*\le y+1)\le C^\prime n^{-1/2}\,,\ y\in\bbr\,, 
\end{equation} 
uniformly in $j$. Indeed, suppose that this is the case. Then  for large
$n$ and every $z\in\bbr$, 
\begin{align*}
&\int_{[n\vep-\beta_{n,j}z,\infty)}\exp\left(-\tilde\theta_{n,j}(s-n\vep)\right) P(S_{n,j}^*\in ds)\\
&\le \sum_{j=1}^\infty e^{-\tilde\theta_{n,j}(j-1-\beta_{n,j}z)}P\left(S_{n,j}^*-n\vep-\beta_{n,j}z\in[j-1,j)\right)\\
&\le C^\prime   n^{-1/2}e^{\tilde\theta_{n,j}\beta_{n,j}z}\left(1-e^{-\tilde\theta_{n,j}}\right)^{-1}\,, 
\end{align*}
which shows \eqref{sm.l2.eq3}.

It remains to prove \eqref{e:claim.un}. We start by observing that for
$M>0$
\begin{align*}
P(S_{n,j}^*>nM)\leq \exp\left\{ n\left[ \tilde \psi_{n,j}(\tilde
  \theta_{n,j}+1)-\tilde\psi_{n,j}(\tilde \theta_{n,j}) -M \tilde
  \theta_{n,j}\right]\right\}. 
\end{align*}
Since the values of both $\tilde \psi_{n,j}(\tilde
  \theta_{n,j}+1)$ and $\tilde\psi_{n,j}(\tilde \theta_{n,j}) $
  remain within a compact set independent of $n$ and $j$, while
  $\theta_{n,j}$ converges, uniformly in $j$,  to $A^{-1}\tau(\vep)>0$, 
 we see that by
 choosing $M$ large enough we can ensure that there is $c>0$ such that
 for all $n$ large  enough, 
 $$
 P(S_{n,j}^*>nM)\leq e^{-cn} \ \ \text{for all $j$.}
 $$
 An identical argument shows that, if $M>0$ is large enough, then
 there is $c>0$ such that
 for all $n$ large  enough, 
 $$
 P(S_{n,j}^*<-nM)\leq e^{-cn} \ \ \text{for all $j$.}
 $$
That means that it suffices to prove that \eqref{e:claim.un} holds for
all $|y|\leq nM$, uniformly in $j$.

Notice that by
part (b) of Theorem \ref{f.cs}, for any $h>0$ there is $C_h^\prime>0$
such that 
\begin{equation} \label{e:claim.know}
P(y\le S_n^*\le y+h)\le C_h^\prime n^{-1/2}\,,\ y\in\bbr,
\end{equation}
where $S_n^*$ is a random variable with the law
\[
P(S_{n}^*\in ds)=\frac1{ E\left(\exp(\theta_{n}S_{n})\right)}e^{\theta_{n}s}P(S_{n}\in ds)\,.
\]
Write 
\begin{align*}
P(y\le S_n^*\le y+h) =&
  \frac{E\left(\exp(\tilde\theta_{n,j}S_{n,j})\right)}{
  E\left(\exp(\theta_{n}S_{n})\right)}
  \frac{1}{E\left(\exp(\tilde\theta_{n,j}S_{n,j})\right)} \\
&\hskip .7in\int_{[y,y+h]}
  \exp\bigl\{  (\theta_n - \tilde \theta_{n,j})s\bigr\}
  e^{\tilde \theta_{n,j}s}P(S_{n}\in ds). 
\end{align*}
By \eqref{sm.l2.eq1}, the factor $\exp\{  (\theta_n - \tilde
\theta_{n,j})s\}$ above is uniformly bounded away from zero over $s\in
[y,y+h]$, $|y|\leq nM$ and $j$. Furthermore,
\begin{align*} 
 \frac{E\left(\exp(\tilde\theta_{n,j}S_{n,j})\right)}{
  E\left(\exp(\theta_{n}S_{n})\right)} 
= \exp\left\{ -n\left[ \psi_n(\theta_n) -\psi_n( \tilde
\theta_{n,j})\right]-\phi_Z(\tilde \theta_{n,j} \beta_{n,j})
\right\},
\end{align*}
and it follows from \eqref{sm.l2.eq2} and uniform boundedness of the
argument of $\phi_Z$ that the ratio above is bounded away from zero
over $n$ and $j$. We conclude that for some $c>0$, for all $n$ large
enough and $|y|\leq nM$,
\begin{align*}
P(y\le S_n^*\le y+h) \geq& c \frac{1}{E\left(\exp(\tilde\theta_{n,j}S_{n,j})\right)}\int_{[y,y+h]}
    e^{\tilde \theta_{n,j}s}P(S_{n}\in ds) \\
\geq& c P(0\leq \beta_{n,j}Z\leq h-1)
   P(y\leq S_{n,j}^*\leq y+1). 
\end{align*}
Since $\beta_{n,j}$ is uniformly bounded, we can choose $h$ large
enough such that $P(0\leq \beta_{n,j}Z\leq h-1)$ is uniformly bounded
away from zero, and \eqref{e:claim.un} follows from
\eqref{e:claim.know}. 
\end{proof}

The next, final, lemma is a major ingredient in the proof of Theorem
\ref{sm.t1}.

\begin{lemma}\label{sm.l3}

For a fixed $k\ge1$, the conditional law of
$(S_n-n\vep,Z_{-k},\ldots$ $\ldots,Z_k,$ $Z_{n-k},\ldots,Z_{n+k})$ given $E_0$ 
 converges weakly, as $n\to\infty$, to the law of  
\[
\left(T^*,Z_{-k}^-,\ldots,Z_k^-,Z_{-k}^+,\ldots,Z_k^+\right)\,.
\]  
\end{lemma}

\begin{proof}
 Consider the following truncated version of $S_n$: 
\begin{align*}
\bar
S_n&=\sum_{j=k+1}^{n-k-1}(A^+_{n-1-j}+A^-_j)Z_j+\sum_{j=n+k}^\infty
(A^-_j-A^-_{j-n}) Z_j\\
&\,\,\,\,\,+\sum_{j=k+1}^\infty(A^+_{j+n-1}-A^+_{j-1})Z_{-j},
\end{align*}
$n\ge2(k+1)$. 
We claim that there exists $c_n>0$ such that for any $x\in\bbr$ and 
any sequence $x_n\to x$, 
\begin{equation}
\label{sm.l3.eq2}P(\bar S_n\ge n\vep+x_n)\sim c_ne^{-x\tau(\vep)/A}\,.
\end{equation}
 To show this we proceed as in the proof of Lemma \ref{sm.l1}. Let 
\[
\bar\psi_n(t)=n^{-1}E\left(e^{t\bar S_n}\right),\ t\in\bbr\,.
\]
Repeating the argument in  Lemma \ref{sm.l1} shows that
\begin{equation}
\label{e:2nd.bar}\lim_{n\to\infty}\bar\psi_n^\pp(t)=A^2\varphi_Z^\pp(At)
\end{equation}
locally uniformly in $t\in\bbr$, and that for large $n$ 
there exists $\bar\theta_n>0$ such that 
\begin{equation}
\label{sm.l3.eq3}\bar\psi_n^\prime(\bar\theta_n)=\vep\,,
\end{equation}
\[
\lim_{n\to\infty}\bar\theta_n=A^{-1}\tau(\vep)\,,
\]
and 
\begin{equation} \label{sm.l3.eq4}
P(\bar S_n\ge n\vep)\sim \frac C{\sqrt
  n}\exp\left(-n(\bar\theta_n\vep-\bar\psi_n(\bar\theta_n))\right),\ n\to\infty\,,
\end{equation}
with $C$  as in Lemma \ref{sm.l1}. The same argument shows that, if
$x_n\to x$, then for large $n$ there exists $\bar\theta_{n,x}>0$ such
that 
\begin{equation}
\label{sm.l3.eq5}\bar\psi_n^\prime(\bar\theta_{n,x})=\vep+n^{-1}x_n\,,
\end{equation}
\[
\lim_{n\to\infty}\bar\theta_{n,x}=A^{-1}\tau(\vep)\,,
\]
and
\begin{equation}
\label{sm.l3.eq6}P(\bar S_n\ge n\vep+x_n)\sim \frac C{\sqrt
  n}\exp\left(-n(\bar\theta_{n,x}(\vep+n^{-1}x_n)-\bar\psi_n(\bar\theta_{n,x}))\right),
\ n\to\infty\,. 
\end{equation}

The mean value theorem applied to \eqref{sm.l3.eq3} and
\eqref{sm.l3.eq5}, together with \eqref{e:2nd.bar} implies that 
\[
A^2\varphi_Z^\pp(\tau(\vep))(\bar\theta_{n,x}-\bar\theta_n)=n^{-1}x_n+o(n^{-1})= n^{-1}x+o(n^{-1})\,.
\]
We use this fact together with another application of the mean value
theorem. Keeping in mind 
\eqref{sm.l3.eq3} and the locally uniform boundedness of the second
derivative implied by \eqref{e:2nd.bar}, we see that 
\[
\bar\psi_n(\bar\theta_{n,x})-\bar\psi_n(\bar\theta_n)=\frac1n\frac{\vep x}{A^2\varphi_Z^\pp(\tau(\vep))}+o(n^{-1})\,.
\]
Putting together the above two displays, we see that
\[
\left(\bar\theta_{n,x}-\bar\theta_n\right)\vep-\bar\psi_n(\bar\theta_{n,x})+\bar\psi_n(\bar\theta_{n})=o(n^{-1})\,,
\]
which in conjunction with  \eqref{sm.l3.eq6} establishes
\eqref{sm.l3.eq2} with $c_n$ given by the right-hand side of
\eqref{sm.l3.eq4}. 

Finally, for $t\geq 0$ and a compact rectangle $R\subset\bbr^{4k+2}$,
\begin{align}
\nonumber&P\left([S_n-n\vep\ge t,\,
 (Z_{-k},\ldots,Z_k,Z_{n-k},\ldots,Z_{n+k})\in R]\cap  E_0\right)\\ 
\nonumber=&P\left(S_n-n\vep\ge t,\, (Z_{-k},\ldots,Z_k,Z_{n-k},\ldots,Z_{n+k})\in R\right)\\
\label{sm.l3.eq7}=&\int_{(x_{-k},\ldots,x_k,y_{-k},\ldots,y_{k})\in R}
P\Bigl(\bar S_n
\ge  n\vep+t-\sum_{j=0}^k(A^+_{n-1-j}+A^-_j)x_{j} \\
\nonumber& 
\hskip.02in-\sum_{j=1}^k(A^+_{j+n-1}
-A^+_{j-1})x_{-j} -\sum_{j=1}^k(A^+_{j-1}+A^-_{n-j})y_{-j} -\sum_{j=0}^k (A^-_{n+j}-A^-_j)y_j\Bigr)\\
\notag &\hskip 1.9in F_Z(dx_{-k})\ldots F_Z(dx_k) 
F_Z(dy_{-k})\ldots F_Z(dy_{k})\\
\nonumber\sim&
               c_ne^{-\tau(\vep)t/A}\int_{(x_{-k},\ldots,x_k,y_{-k},\ldots,y_{k})\in  R}
               \exp\Biggl\{\frac{\tau(\vep)}A\Biggl(\sum_{j=0}^k(A^++A^-_j)x_j
              \\
 \nonumber& \hskip .02in  +\sum_{j=1}^k(A^+-A^+_{j-1})x_{-j}
+\sum_{j=1}^k(A^+_{j-1}+A^-)y_{-j} + \sum_{j=0}^k (A^--A^-_j)y_j
\Biggr)\Biggr\}  \\
\notag &\hskip 1.9in
F_Z(dx_{-k})\ldots F_Z(dx_k)F_Z(dy_{-k}) \ldots
               F_Z(dy_{k}) 
\end{align}
as $n\to\infty$.  In order to justify the asymptotic equivalence above,
note that for each fixed $x_{-k},\ldots,x_k,y_{-k},\ldots,y_{k}$,
$c_n^{-1}$ times the integrand of \eqref{sm.l3.eq7} converges, by
\eqref{sm.l3.eq2}, to  
\begin{align*}
&\exp\Biggl\{\frac{\tau(\vep)}A\Biggl(-t+\sum_{j=0}^k(A^++A^-_j)x_j
+\sum_{j=1}^k(A^+-A^+_{j-1})x_{-j}\\
&\hskip 1.9in +\sum_{j=1}^k(A^+_{j-1}+A^-)y_{-j} 
+ \sum_{j=0}^k (A^--A^-_j)y_j\Biggr)\Biggr\}\,. 
\end{align*}
The absolute value of each of the variables
$x_{-k},\ldots,x_k,y_{-k},\ldots,y_{k}$ in the rectangle $R$ has a
finite upper bound. Replacing each of these variables
 by the corresponding upper bound  of its absolute value 
and using \eqref{sm.l3.eq2} once again, provides
a bound to use in the dominated convergence theorem.

Continuing to keep $k$ an arbitrary fixed positive integer, we claim that, as $n\to\infty$, 
\begin{align} \label{e:E0.int}
P(E_0)&\sim c_n\int_{\bbr^{4k+2}}
        \exp\Biggl\{\frac{\tau(\vep)}A\Biggl( 
\sum_{j=0}^k(A^++A^-_j)x_j
+\sum_{j=1}^k(A^+-A^+_{j-1})x_{-j}\\
\notag &\hskip 1.5in +\sum_{j=1}^k(A^+_{j-1}+A^-)y_{-j} 
+ \sum_{j=0}^k (A^--A^-_j)y_j\Biggr)\Biggr\} \\
\nonumber&\hskip 1.5 in F_Z(dx_{-k})\ldots F_Z(dx_k)F_Z(dy_{-k})\ldots
           F_Z(dy_{k}). 
\end{align}
Once this has been established, the claim of the lemma will follow from
\eqref{sm.l3.eq7} and \eqref{e:E0.int}, completing the argument. To prove
\eqref{e:E0.int}, we notice that by \eqref{sm.l3.eq2} and Fatou's
lemma,
\begin{align} 
\notag&\liminf_{n\to\infty} c_n^{-1} P(E_0)\\
\label{e:fatou}&\geq 
\int_{\bbr^{4k+2}}
        \exp\Biggl\{\frac{\tau(\vep)}A\Biggl( 
\sum_{j=0}^k(A^++A^-_j)x_j
+\sum_{j=1}^k(A^+-A^+_{j-1})x_{-j}\\
\notag &\hskip 1in +\sum_{j=1}^k(A^+_{j-1}+A^-)y_{-j} 
+ \sum_{j=0}^k (A^--A^-_j)y_j\Biggr)\Biggr\} \\
\nonumber&\hskip 1 in F_Z(dx_{-k})\ldots F_Z(dx_k)F_Z(dy_{-k})\ldots
           F_Z(dy_{k}). 
\end{align}
By Lemma \ref{sm.l2} the sequence of the
conditional laws of 
 $(Z_{-k},\ldots,Z_k,Z_{n-k},\ldots$ $\ldots,Z_{n+k})$ given
$E_0$ is tight in $\bbr^{4k+2}$. Let $\nu$ be a subsequential limit of
this sequence. It follows from the above inequality and
\eqref{sm.l3.eq7}  with $t=0$ that
$$
\nu(R)\leq P\bigl( (Z_{-k}^-,\ldots,Z_k^-,Z_{-k}^+,\ldots,Z_{k}^+)\in
R\bigr)
$$
for any compact rectangle $R$ in $\bbr^{4k+2}$, which can only happen
if $\nu$ is, in fact, the law of the random vector $
(Z_{-k}^-,\ldots,Z_k^-,Z_{-k}^+,\ldots,Z_{k}^+)$. Therefore,
\eqref{e:fatou} must hold as an equality. 
\end{proof}

We are now ready to prove the first of our main theorems.

\begin{proof}[Proof of Theorem \ref{sm.t1}]
We start with showing that for every fixed $k\ge1$, conditionally on
$E_0$ as $n\to\infty$, 
\begin{align}
\notag&\left(S_n-n\vep,X_{-k},\ldots,X_k,X_{n-k},\ldots,X_{n+k}\right)\\
\label{sm.t1.eq1}
&\Rightarrow
(T^*,U_{-k}^-,\ldots,U_k^-,U_{-k}^+,\ldots,U_k^+)\,,
\end{align}
with $(U_k^-)$ and $(U_k^+)$ defined in \eqref{e:U}. 
For all $i\ge1$ let 
\[
X_m^{(i)}=\sum_{j=-i}^ia_jZ_{m-j}, \ m\in\bbz\,.
\]
Lemma \ref{sm.l3} implies that for a fixed $i$,
\[
\left(S_n-n\vep,X_{-k}^{(i)},\ldots,X_k^{(i)},X_{n-k}^{(i)},\ldots,X_{n+k}^{(i)}\right)
\]
converges weakly as $n\to\infty$, conditionally on $E_0$, to
\[
 (T^*,U_{-k}^{-(i)},\ldots,U_k^{-(i)},U_{-k}^{+(i)},\ldots,U_k^{+(i)})\,,
\]
where
\[
U_m^{\pm(i)}=\sum_{j=-i}^ia_jZ_{m-j}^\pm, \ m\in\bbz\,.
\]

Note that by  Lemma \ref{sm.l2}, for every $\delta>0$,   
\[
\sup_{n\geq 1}\sup_{m\in\bbz}
P\left(\left|X_m^{(i)}-X_m\right|>\delta|E_0\right)\le\frac1\delta\left[\sup_{n\ge
    1,\, m\in\bbz}E(|Z_m||E_0)\right]\sum_{|j|>i} |a_j|\to0 
\]
as $i\to\infty$.  Since the two series in \eqref{e:U}
converge in probability,  the claim
\eqref{sm.t1.eq1} follows from Theorem 3.2 in
\cite{billingsley:1999}. 

Notice that 
\begin{align*}
\one(E_j)&=\one\left(\sum_{i=j}^{j+n-1}X_i\ge n\vep\right)\\
         &= \left\{ \begin{array} {ll}
        \one\left(S_n-n\vep -\sum_{i=0}^{j-1}X_i
           +\sum_{i=n}^{n+j-1}X_i\geq 0\right) & \text{if} \ j\geq 0,\\
       \one\left(S_n-n\vep + \sum_{i=j}^{-1} X_i - \sum_{i=n+j}^{n-1} X_i\geq 0\right) & \text{if} \ j
<0.
\end{array}\right.
\end{align*}

We conclude by \eqref{sm.t1.eq1} and the continuous mapping theorem
that for $K\geq 1$,

\begin{align*}
&\bigl( \one(E_j), \, j=-K,\ldots, K\bigr) \\
&\Rightarrow\Biggl( \one\biggl(T^*+\sum_{i=j}^{-1} U_i^--\sum_{i=j}^{-1} U_i^+\geq
                                           0\biggr),\,  j=-K,\ldots,
                                            -1, \\
&\hskip 0.4in \one\biggl(T^*-\sum_{i=0}^{j-1}U_i^- +
                                          \sum_{i=0}^{j-1} U_i^+\geq
                                           0\biggr),\,  j=1,\ldots, K\Biggr)\\
  &=\bigl( V_j(\vep), \, j=-K,\ldots, K\bigr) 
\end{align*}
as $n\to\infty$, where the law of the vector in the left hand side is
computed conditionally  on $E_0$. Indeed, the continuity of the exponential random
variable $T^*$ means that the boundary of the $K$-dimensional set
above has limiting probability zero. This proves \eqref{e:seq.conv}. 
\end{proof}

Finally, we prove our second main result.

\begin{proof}[Proof of Theorem \ref{sm.t2}]
  We take the stochastic process
  $$
  W^+_\vep(t)=
  \frac{\tau(\vep)}A\sum_{i=0}^{[t\vep^{-2}]}(U_i^--U_i^+), \, t\geq 0
  $$
  and its version in the direction of the negative time and prove for
  them   joint functional weak convergence. The claim of the theorem
  will then follow by an application of the continuous mapping
  theorem.   We start with some variance calculations. For a large $m$,
  \begin{align*}
\var\left( \sum_{i=0}^m U_i^-\right) =& \var \left( \sum_{i=0}^m
    \sum_{k=-\infty}^{-1} a_{i-k}Z_k^{-}\right) +  \var \left( \sum_{i=0}^m
    \sum_{k=0}^{\infty} a_{i-k}Z_k^{-}\right) \\
    =& \sum_{k=1}^\infty \left( \sum_{i=k}^{m+k}a_i\right)^2
       \var(Z_{-k}^-)
       + \sum_{k=0}^{m} \bigl(A^+_{m-k}+A_k^-\bigr)^2\var(Z_k^-)\\
+& \sum_{k=m+1}^\infty \left( \sum_{i=-k}^{m-k}a_i\right)^2 \var(Z_k^-).
  \end{align*}
  It is elementary that
  $$
  \var(Z_{k}^-) \to\sigma_Z^2 \ \text{as} \  \vep\to 0 
  $$
  uniformly in $k\in\bbz$. Therefore,
  $$
  \sum_{k=0}^{m}  \bigl(A^+_{m-k}+A_k^-\bigr)^2
\var(Z_k^-)\sim \sigma_Z^2 \sum_{k=0}^{m}
\bigl(A^+_{m-k}+A_k^-\bigr)^2 \sim m\sigma_Z^2A^2
 $$
  as $\vep\to 0, \, m\to\infty$. Furthermore,
 $$
  \sum_{k=1}^\infty \left( \sum_{i=k}^{m+k}a_i\right)^2
  \var(Z_{-k}^-) \sim \sigma_Z^2 \sum_{k=1}^\infty \left(
    \sum_{i=k}^{m+k}a_i\right)^2 = o(m)
 $$
   as $\vep\to 0, \, m\to\infty$, with the last statement an easy
   consequence of the absolute summability of $(a_i)$.  Similarly, 
   $$
  \sum_{k=m+1}^\infty \left( \sum_{i=-k}^{m-k}a_i\right)^2 \var(Z_k^-) =  o(m)
 $$
   as $\vep\to 0, \, m\to\infty$.

The   same  argument shows that we also have 
 $$
   \sum_{k=0}^{m} \bigl(A^+_{m-k}+A_k^-\bigr)^2\var(Z_k^+)\sim \sigma_Z^2 \sum_{k=0}^{m}
  \bigl(A^+_{m-k}+A_k^-\bigr)^2 \sim m\sigma_Z^2A^2,
$$
$$
  \sum_{k=1}^\infty \left( \sum_{i=k}^{m+k}a_i\right)^2
  \var(Z_{-k}^+) \sim \sigma_Z^2 \sum_{k=1}^\infty \left(
    \sum_{i=k}^{m+k}a_i\right)^2 = o(m),
  $$
  $$
\sum_{k=m+1}^\infty \left( \sum_{i=-k}^{m-k}a_i\right)^2 \var(Z_k^+) =  o(m)
 $$
   as $\vep\to 0, \, m\to\infty$.  We  write 
\begin{align} \label{e:W}
W^+_\vep(t)=& \ W_\vep(t)\\
=& \ \frac{\tau(\vep)}A\sum_{i=0}^{[t\vep^{-2}]}(U_i^--U_i^+) \\
\notag =&\left[\frac{\tau(\vep)}A\sum_{i=0}^{[t\vep^{-2}]}
          \sum_{k=-\infty}^{-1} a_{i-k}Z_k^{-} -
\frac{\tau(\vep)}A\sum_{i=0}^{[t\vep^{-2}]}
          \sum_{k=-\infty}^{-1} a_{i-k}Z_k^{+} \right]   \\      
\notag &+\left[\frac{\tau(\vep)}A\sum_{i=0}^{[t\vep^{-2}]}
          \sum_{k=0}^{[t\vep^{-2}]} a_{i-k}Z_k^{-} -
\frac{\tau(\vep)}A\sum_{i=0}^{[t\vep^{-2}]}
          \sum_{k=0}^{[t\vep^{-2}]} a_{i-k}Z_k^{+} \right]   \\      
\notag &+\left[\frac{\tau(\vep)}A\sum_{i=0}^{[t\vep^{-2}]}
          \sum_{k=[t\vep^{-2}] +1}^{\infty} a_{i-k}Z_k^{-} -
\frac{\tau(\vep)}A\sum_{i=0}^{[t\vep^{-2}]}
          \sum_{k=[t\vep^{-2}] +1}^{\infty} a_{i-k}Z_k^{+} \right]   \\  
\notag =:& \ \ \ W_\vep^{(1)}(t) + W_\vep^{(2)}(t) + W_\vep^{(3)}(t), \ \vep>0,\, t\ge0. 
\end{align}
For typographical convenience we will omit the superscript in $W^+_\vep$ for now
and bring it back at a certain point later on. 
The assumption $EZ=0$ and \eqref{sm.eqtau} imply that, as $\vep\to0$,
\begin{equation}
\label{sm.t2.eq0}\vep/ A\sim\tau(\vep)\varphi_Z^\pp(0)=\sigma_Z^2\tau(\vep)\,.
\end{equation}
We have, therefore, verified that $\var\bigl(
W_\vep^{(j)}(t)\bigr)\to 0$ as $\vep\to 0$ for every $t$ and $j=1,3$, 
so for every such $j$, 
\begin{equation} \label{e:W1}
W_\vep^{(j)}(t)-E\bigl(
W_\vep^{(j)}(t)\bigr)\to 0 \ \text{in probability as} \ \vep\to 0.
\end{equation}
Furthermore,   for every $t$, as
$\vep\to 0$, 
\begin{align*}
\Var(W_\vep^{(2)}(t)) 
                 \to\frac {2t}{A^2\sigma_Z^2}\,.
\end{align*}
A similar calculation shows that for $0\le s\le t$,
\[
\lim_{\vep\to0}\Cov(W_\vep^{(2)}(s),W_\vep^{(2)}(t))=\frac {2s}{A^2\sigma_Z^2}\,.
\]

Observe next that the third absolute moment of both $(Z_k^-)$ and
$(Z_k^+)$ is bounded uniformly in $\vep$ and $k$. Therefore, the
Lindeberg condition is satisfied by the triangular array defined by
any finite linear combination of the type $\theta_1
W_\vep^{(2)}(t_1)+\ldots + \theta_d W_\vep^{(2)}(t_d)$. Applying the
Lindeberg Central Limit Theorem (see e.g. Theorem 27.2 in
\cite{billingsley:1995}) and the Cram\'er-Wold device we conclude that
the finite dimensional distributions of
$W_\vep^{(2)}(t)-E(W_\vep^{(2)}(t))$ converge to those of
$(A\sigma_Z)^{-1}\sqrt2B_t$, where $B_t$ is a standard Brownian
motion. It follows from \eqref{e:W1} that the finite dimensional distributions of
$W_\vep(t)-E(W_\vep(t))$ converge to the same limit.

Next, let $0\leq s<t$. If $\vep^2>(t-s)$, then for any $s\leq r\leq t$
either
$$
W_\vep(t)-E(W_\vep(t))= W_\vep(r)-E(W_\vep(r)) \ \text{a.s.}
$$
or
$$
W_\vep(s)-E(W_\vep(s))= W_\vep(r)-E(W_\vep(r)) \ \text{a.s.}, 
$$
so that 
\begin{align*}
&  E\Biggl[\Bigl(W_\vep(t)-W_\vep(r)-E\bigl(W_\vep(t)-W_\vep(r)\bigr)\Bigr)^2\\
&\hskip 1in \Bigl(W_\vep(r)-W_\vep(s)-E\bigl(W_\vep(r)-W_\vep(s)\bigr)\Bigr)^2\Biggr]=0. 
\end{align*}
Suppose now that $\vep^2\leq (t-s)$. We have
\begin{align*}
&E\left[\Bigl(W_\vep(t)-W_\vep(s)-E\bigl(W_\vep(t)-W_\vep(s)\bigr)\Bigr)^4\right] \\
\leq
 12 &E\left[\Bigl(W_\vep^{(1)}(t)-W_\vep^{(1)}(s)-E\bigl(W_\vep^{(1)}(t)-W_\vep^{(1)}(s)\bigr)\Bigr)^4\right] \\
+   12
  &E\left[\Bigl(W_\vep^{(2)}(t)-W_\vep^{(2)}(s)-E\bigl(W_\vep^{(2)}(t)-W_\vep^{(2)}(s)\bigr)\Bigr)^4\right]\\
 +   12
  &E\left[\Bigl(W_\vep^{(3)}(t)-W_\vep^{(2)}(s)-E\bigl(W_\vep^{(3)}(t)-W_\vep^{(2)}(s)\bigr)\Bigr)^4\right].
\end{align*}
For a positive constant $C$ independent of $\vep,s,t$, that may change
from appearance to appearance, 
\begin{align} 
\notag &E\Biggl[\Biggl(\frac{\tau(\vep)}{A}\Bigl(\sum_{i=0}^{[t\vep^{-2}]}
          \sum_{k=0}^{[t\vep^{-2}]}  a_{i-k}\bigl( Z_k^{-} -E(Z_k^{-}
                 )\bigr)\\
\notag                 &\hskip 1in - \sum_{i=0}^{[s\vep^{-2}]}
          \sum_{k=0}^{[s\vep^{-2}]}  a_{i-k}\bigl( Z_k^{-} -E(Z_k^{-}
                 )\bigr) \Bigr)
\Biggr)^4\Biggr]\\
\label{e:two.terms} \leq& \ C\vep^4 E\left(\sum_{k=0}^{[s\vep^{-2}]}\bigl( Z_k^{-} -E(Z_k^{-}
                 )\bigr) \sum_{i=[s\vep^{-2}]-k}^{[t\vep^{-2}]-k} a_i\right)^4\\
\notag & +C\vep^4 E\left( \sum_{k=[s\vep^{-2}]+1}^{[t\vep^{-2}]} \bigl( Z_k^{-} -E(Z_k^{-}
                 )\bigr) \sum_{i=-k}^{[t\vep^{-2}]-k} a_i\right)^4.
\end{align}
Since the
fourth moments of $(Z_k^-)$ are bounded uniformly in $\vep$ and $k$,
and the coefficients $(a_i)$ are absolutely summable, the first term
in the right hand side can be bounded by 
\begin{align*}
 & C\vep^4  \sum_{k=0}^{[s\vep^{-2}]} E\bigl( Z_k^{-} -E(Z_k^{-}
  )\bigr)^4 \left( \sum_{i=[s\vep^{-2}]-k}^{[t\vep^{-2}]-k} |a_i|
\right)^4\\
 +& \left[ C\vep^2 \sum_{k=0}^{[s\vep^{-2}]} E\bigl( Z_k^{-} -E(Z_k^{-}
  )\bigr)^2 \left( \sum_{i=[s\vep^{-2}]-k}^{[t\vep^{-2}]-k} |a_i|
\right)^2\right]^2\\
  \leq& C\vep^4  \sum_{k=0}^\infty \sum_{i=-\infty}^\infty 
       |a_{i-k}|+\left[C\vep^2  \sum_{k=0}^\infty  \sum_{i=[s\vep^{-2}]-k}^{[t\vep^{-2}]-k} |a_i|
\right]^2.
\end{align*}
Since
\begin{align*}
&\sum_{k=0}^\infty \sum_{i=[s\vep^{-2}]-k}^{[t\vep^{-2}]-k} |a_i|
  \leq \bigl( [t\vep^{-2}]-[s\vep^{-2}]\bigr) \sum_{i=-\infty}^\infty |a_i|\\
  \leq& \bigl(t\vep^{-2}-(s\vep^{-2}-1)\bigr) \sum_{i=-\infty}^\infty |a_i|
        \leq 2(t-s)\vep^{-2}\sum_{i=-\infty}^\infty |a_i|, 
\end{align*}
we conclude that the first term in the right hand side of
\eqref{e:two.terms} is bounded by $ C(t-s)^2$. In a similar way one can
show that the second term in the right hand side of
\eqref{e:two.terms} is also bounded by $ C(t-s)^2$. 

The same argument shows that
\begin{align*}
&E\Biggl[\Biggl(\frac{\tau(\vep)}{A}\Bigl(\sum_{i=0}^{[t\vep^{-2}]}
          \sum_{k=0}^{[t\vep^{-2}]}  a_{i-k}\bigl( Z_k^{+} -E(Z_k^{+}
                 )\bigr)\\
                 &\hskip 1in  - \sum_{i=0}^{[s\vep^{-2}]}
          \sum_{k=0}^{[s\vep^{-2}]}  a_{i-k}\bigl( Z_k^{+} -E(Z_k^{+}    )\bigr) \Bigr)
\Biggr)^4\Biggr] \\
&\leq C(t-s)^2,
\end{align*}
so that
$$
E\left[\Bigl(W_\vep^{(2)}(t)-W_\vep^{(2)}(s)-E\bigl(W_\vep^{(2)}(t)-W_\vep^{(2)}(s)\bigr)\Bigr)^4\right] 
\leq C(t-s)^2. 
$$
In the same way we can check that 
$$
E\left[\Bigl(W_\vep^{(j)}(t)-W_\vep^{(1)}(s)-E\bigl(W_\vep^{(j)}(t)-W_\vep^{(1)}(s)\bigr)\Bigr)^4\right] 
\leq C(t-s)^2, \ j=1,3,
$$
so we conclude that
\begin{equation} \label{e:4th}
E\left[\Bigl(W_\vep(t)-W_\vep(s)-E\bigl(W_\vep(t)-W_\vep(s)\bigr)\Bigr)^4\right] 
\leq C(t-s)^2 
\end{equation}
 if $\vep^2\leq (t-s)$. By the Cauchy-Schwarz inequality, for any
 $0\leq s\leq r\leq t$ we have
 \begin{align*}
  &E\Bigl[\Biggl(W_\vep(t)-W_\vep(r)-E\bigl(W_\vep(t)-W_\vep(r)\bigr)\Bigr)^2\\
&\hskip 1in \Bigl(W_\vep(r)-W_\vep(s)-E\bigl(W_\vep(r)-W_\vep(s)\bigr)\Bigr)^2\Biggr] \\
&   \leq C(t-s)^2 
\end{align*}
if $\vep^2\leq (t-s)$. Appealing to Theorem 13.5 of
\cite{billingsley:1999}  we conclude that for any fixed $T$, the family 
\[
\left\{\left(W_\vep(t)-E(W_\vep(t)):0\le t\le T\right):\vep>0\right\}
\]
is tight in $D[0,T]$ endowed with the Skorohod  $J_1$
topology. Therefore, as $\vep\to0$, 
\begin{equation}
\label{sm.t2.eq1}\left(W_\vep(t)-E(W_\vep(t)):0\le t\le T\right)\Rightarrow\left((A\sigma_Z)^{-1}\sqrt2B_t:0\le t\le T\right)\,,
\end{equation}
in $D[0,T]$. Furthermore,
\begin{align*}
&E\left(W_\vep(t)\right) \\
=& 
\frac{\tau(\vep)}{A}\Biggl[\sum_{i=0}^{[t\vep^{-2}]}
                             \sum_{k=-\infty}^{-1} a_{i-k}Z_k^{-}
                             +\sum_{i=0}^{[t\vep^{-2}]}
          \sum_{k=0}^{[t\vep^{-2}]} a_{i-k}Z_k^{-} 
+ \sum_{i=0}^{[t\vep^{-2}]}
          \sum_{k=[t\vep^{-2}] +1}^{\infty} a_{i-k}Z_k^{-}
\\
&\hskip 0.3in -\sum_{i=0}^{[t\vep^{-2}]}
                             \sum_{k=-\infty}^{-1} a_{i-k}Z_k^{+}
-\sum_{i=0}^{[t\vep^{-2}]}
          \sum_{k=0}^{[t\vep^{-2}]} a_{i-k}Z_k^{+}  
  - \sum_{i=0}^{[t\vep^{-2}]}
          \sum_{k=[t\vep^{-2}] +1}^{\infty} a_{i-k}Z_k^{+} \Biggr]
   \end{align*}
  Clearly, 
  $|EZ_k^{\pm}|=O(\tau(\vep))$ uniformly in $\vep$ and
  $k\in\bbz$. Therefore,
  \begin{align*}
\left| \sum_{i=0}^{[t\vep^{-2}]}
    \sum_{k=-\infty}^{-1} a_{i-k}EZ_k^{-} \right| \leq O(\tau(\vep))
  \sum_{i=0}^{[t\vep^{-2}]}\sum_{k=i+1}^\infty   |a_{k}|
    =o\bigl(\tau(\vep)\vep^{-2}\bigr) 
  \end{align*}
  uniformly in $t$ in a compact set. Similarly,
 \begin{align*}
\left| \sum_{i=0}^{[t\vep^{-2}]}
       \sum_{k=[t\vep^{-2}] +1}^{\infty} a_{i-k}Z_k^{-}
 \right|=o\bigl(\tau(\vep)\vep^{-2}\bigr), 
\end{align*}
and by the same argument, 
   \begin{align*}
\left| \sum_{i=0}^{[t\vep^{-2}]}
    \sum_{k=-\infty}^{-1} a_{i-k}EZ_k^{+} \right|  
    &=o\bigl(\tau(\vep)\vep^{-2}\bigr), \\
 \left| \sum_{i=0}^{[t\vep^{-2}]}
       \sum_{k=[t\vep^{-2}] +1}^{\infty} a_{i-k}Z_k^{+}
 \right|&=o\bigl(\tau(\vep)\vep^{-2}\bigr),
  \end{align*}
 all uniformly in $t$ in a compact set. 

Finally, $EZ_k^-\sim \tau(\vep)\sigma_Z^2(A^++A_k^-)/A$
  as $\vep\to 0$ uniformly in $k \geq 0$, so 
  \begin{align*}
  \sum_{i=0}^{[t\vep^{-2}]}
    \sum_{k=0}^{[t\vep^{-2}]}  a_{i-k}EZ_k^{-} \sim A\tau(\vep)\sigma_Z^2t\vep^{-2}.
  \end{align*}
  Similarly, $EZ_k^+\sim \tau(\vep)\sigma_Z^2(A^--A_k^-)/A$
  as $\vep\to 0$ uniformly in $k \geq 0$, so
  \begin{align*}
  \sum_{i=0}^{[t\vep^{-2}]}
    \sum_{k=0}^{[t\vep^{-2}]}  a_{i-k}EZ_k^{+}  =o\bigl(\tau(\vep)\vep^{-2}\bigr),
  \end{align*}
all  uniformly in $t$ in a compact set. We conclude by \eqref{sm.t2.eq0}
 that for all $\vep>0$ small enough,
 \begin{align} \label{e:mean.l}
E\left(W_\vep(t)\right)\geq \frac t{2A^2\sigma_Z^2}, \ t\geq 1
 \end{align}
 and
\begin{align*}
E\left(W_\vep(t)\right)& \to\frac t{A^2\sigma_Z^2}, \ \vep\to0,
\end{align*}
uniformly in $t$ in a compact set. Since the addition in $D[0,T]$ is
continuous at continuous functions, this 
 along with \eqref{sm.t2.eq1} shows that
\begin{equation} \label{e:weak.T}
\left(W^+_\vep(t):0\le t\le
  T\right)\Rightarrow\left((A\sigma_Z)^{-1}\sqrt2B^+_t+(A\sigma_Z)^{-2}t:0\le
  t\le T\right), 
\end{equation}
in $D[0,T]$ as $\vep\to0$. 
Notice that we have brought back the superscript in
$W^+_\vep$ omitted since \eqref{e:W}, and we have also added a
superscript to the standard Brownian motion $B^+$.

Clearly we can also define 
\begin{align} \label{e:Wmin}
W^-_\vep(t)=&\frac{\tau(\vep)}A\sum^{-1}_{i=-[t\vep^{-2}]}(U_i^+-U_i^-),
              \, t\geq 0, 
 \end{align}
and use the same argument (or, even simpler, just appeal to time
inversion) to show that for any $T>0$, 
\begin{equation} \label{e:weak.Tmin}
\left(W^-_\vep(t):0\le t\le
  T\right)\Rightarrow\left((A\sigma_Z)^{-1}\sqrt2B^-_t+(A\sigma_Z)^{-2}t:0\le
  t\le T\right), 
\end{equation}
in $D[0.T]$ as $\vep\to0$, 
where $B^-$ is a standard Brownian motion.

We claim that, in fact, we have joint convergence
\begin{align} \label{e:weak.Tjoint}
&\Bigl(\bigl(W^+_\vep(t),:0\le t\le T\bigr) ,\, \bigl(W^-_\vep(t),:0\le t\le T\bigr)\Bigr)\\
\notag \Rightarrow&\Bigl(\bigl((A\sigma_Z)^{-1}\sqrt2 B^+_t+(A\sigma_Z)^{-2}t:0\le
  t\le T\bigr), \\
\notag &\hskip 0.1in \bigl((A\sigma_Z)^{-1}\sqrt2 B^-_t+(A\sigma_Z)^{-2}t:0\le
  t\le T\bigr)\Bigr), 
\end{align}
in $D[0.T]\times D[0.T]$ as $\vep\to0$, where the standard Brownian
motions in the right hand side are independent. To see this, recall
that the only term in \eqref{e:W} that contributes to the randomness
in the weak limit of $\bigl(W^+_\vep(t),:0\le t\le T\bigr)$ is the term
$\bigl(W^{(2)}_\vep(t),:0\le t\le T\bigr)$, which is a function of $(Z_k^\pm)$
with $k\geq 0$. An identical argument shows that the only term in the
same expansion of $\bigl(W^-_\vep(t),:0\le t\le T\bigr)$ that contributes to the
randomness in the limit  is a function of $(Z_k^\pm)$
with $k<0$. Since the random variables $(Z_k^\pm)$ are independent, we
obtain the claimed joint weak convergence in \eqref{e:weak.Tjoint},
with independent components in the limit.

For any real $\lambda$ the function $\varphi:\, D[0,T]\to\bbr$ defined
by
$$
\varphi(f) = \int_0^T \one\bigl(\lambda\geq f(t)\bigr)\, dt
$$
is continuous at any continuous $f$ that takes  value $\lambda$
only on a set of measure 0.  Therefore, for any such $\lambda$, by the
continuous mapping theorem,
\begin{align*}
  \int_0^T \one\bigl( \lambda\geq W^\pm_\vep(t)\bigr)\, dt \Rightarrow&
  \int_0^T \one\bigl(\lambda\ge
    (A\sigma_Z)^{-1}\sqrt2B_t^\pm+(A\sigma_Z)^{-2}t\bigr)\, dt 
\\
\eid&\int_0^T
     \one\bigl(\lambda\ge\sqrt2B^\pm_{(A\sigma_Z)^{-2}t}+(A\sigma_Z)^{-2}t\bigr)\,
     dt   \\
=& \ A^2\sigma_Z^2\int_0^{(A\sigma_Z)^2T}\one\bigl(\lambda\ge\sqrt2B^\pm_t+t\bigr)\,
               dt\,.
\end{align*}

Noticing that we can write  
\[
V_j^+(\vep)=\one\left(T_0\ge\frac{\tau(\vep)}A\sum_{i=0}^{j-1}(U_i^- -
  U_i^+)\right), \ j\geq 1,
\]
where $T_0$ is a standard exponential random variable independent of
the collection $(Z_j^u:j\in\bbz,u=+\text{ or }-)$,  we conclude that
for any $T>0$, 
\begin{align*}
\vep^2\sum_{j=0}^{[T\vep^{-2}]}V^+_{j}(\vep) 
&=\int_0^T \one\bigl(T_0\ge W^+_\vep(t)\bigr) \, dt -
     V^+_{[T\vep^{-2}]}\bigl(  T-\vep^2  [T\vep^{-2}]\bigr) 
\\
&\Rightarrow
     A^2\sigma_Z^2\int_0^{(A\sigma_Z)^2T}\one\bigl(T_0\ge\sqrt2B_t^++t)\,
     dt. 
\end{align*}

It is clear that the latter integral converges a.s., as $T\to\infty$, to the integral
prescribed in the theorem. Therefore, we can appeal to Theorem 3.2 in
\cite{billingsley:1999}, which requires us to show  that for
any $\delta>0$,
\begin{equation}
\label{sm.t2.eq2}\lim_{T\to\infty}\limsup_{\vep\to0}P\left(\vep^2\sum_{j=[T\vep^{-2}]+1}^\infty V^+_j(\vep)>\delta\right)=0\,.
\end{equation}
However, by Markov's inequality 
\begin{align*}
&P\left(\vep^2\sum_{j=[T\vep^{-2}]+1}^\infty
  V^+_j(\vep)>\delta\right) 
\leq \vep^2\delta^{-1}\sum_{j=[T\vep^{-2}]+1}^\infty P\bigl( T_0\geq
      W^+_\vep((j-1)\vep^2)\bigr)\\
      \leq&  \vep^2\delta^{-1}\sum_{j=[T\vep^{-2}]+1}^\infty P\left(
      W^+_\vep((j-1)\vep^2)\leq
      \frac{(j-1)\vep^2}{4A^2\sigma_Z^2}\right)\\
 +&  \vep^2\delta^{-1}\sum_{j=[T\vep^{-2}]+1}^\infty\exp\left\{
    -\frac{(j-1)\vep^2}{4A^2\sigma_Z^2}\right\}. 
\end{align*}
By \eqref{e:mean.l} and \eqref{e:4th} we have for some positive
constant $C$, 
$$
P\left(
      W^+_\vep((j-1)\vep^2)\leq
      \frac{(j-1)\vep^2}{4A^2\sigma_Z^2}\right) \leq
    C\vep^{-4}(j-1)^{-2}
    $$
for all $j>[T\vep^{-2}]$, $T\geq 1$ and $\vep>0$ small enough. This
estimate suffices to establish \eqref{sm.t2.eq2}.

Note that this argument also shows that that for small $\vep>0$,
$ED_\vep^+<\infty$, so $D_\vep^+<\infty$ a.s.. 

Since a similar argument can be applied to
$\vep^2\sum_{j=0}^{[T\vep^{-2}]}V^-_{j}(\vep)$, 
the proof is complete. 
\end{proof}


\section{Discussion} \label{sec:discussion}

As is usually the case with large deviations, the limiting
distributions obtained in \eqref{e:seq.conv}  and \eqref{sm.t1.claim}
of Theorem \ref{sm.t1} depend largely on the underlying model through the
distribution of the noise variables $F_Z$ and the coefficients
$(a_i)$. This dependence largely disappears in Theorem \ref{sm.t2}
where the limiting distribution depends only on the noise variance
$\sigma_Z^2$ and the sum of the coefficients $A$. This can be
understood by viewing the case of a small overshoot $\vep$ as
approaching the regime of moderate deviations. Indeed, in the case of
moderate deviations one expects that the central limit behaviour
becomes visible and leads to a collapse of the model ingredients
necessary to describe the limit to a bare minimum consisting of second
order information.

This naturally leads to the question of a difference of how large
deviations cluster between the short memory moving average processes
and long memory moving average processes. It is common to say that the coefficients of the
moving average process \eqref{e:MA} with long memory are square summable but not
absolutely summable. Assuming certain regularity of the coefficients
$(a_i)$ (e.g. their regular variation), one can show that
 for any fixed $j\ge1$ and $\vep>0$,
\[
\lim_{n\to\infty}P\left(E_j(n,\vep)\bigr|E_0(n,\vep)\right)=1\,,
\]
so one expects infinitely many events $(E_j(n,\vep))$ to happen once
$E_0(n,\vep)$ does. This necessitates different limiting procedures when
studying large deviations clustering of such long memory processes. It
is important to note that for these long memory moving average
processes, in the notation of \eqref{eq.defsn}, $n=o(\var(S_n))$, so
one can view the events $(E_j(n,\vep))$ as moderate deviation events
and not large deviation events. Indeed, it turns out that a natural
limiting procedure leads to a collapse in the amount of information
about the model needed to describe the limit, which is similar to the
situation with Theorem \ref{sm.t2} in the present, short memory
case.  This is described in details in \cite{chakrabarty:samorodnitsky:2023}.

\bigskip

{\bf Acknowledgments} Two anonymous referees provided the authors with
incredibly useful and insightful comments. We are very grateful.

\appendix

\section{Some useful facts} \label{sec:appendix}

The following non-logarithmic version
of a large deviation statement and a related estimate are from \cite{chaganty:sethuraman:1983}. 
\begin{theorem}\label{f.cs}
Let $\{T_n\}$ be a sequence of random variables with 
\[
E\left(e^{zT_n}\right)<\infty \  \text{for any} \  z\in\bbr, \
n\ge1\,. 
\]
For a sequence $\{a_n\}$ of  positive numbers with
\begin{equation}
\label{cs.eq1}\lim_{n\to\infty}a_n=\infty
\end{equation}
we denote 
\[
\psi_n(z)=a_n^{-1}\log E\left(e^{zT_n}\right), \ z\in\bbr, \ n\ge1\,.
\]

Let $\{m_n\}$ be a bounded sequence of real numbers. 
Assume that there exists a bounded positive sequence $\{\tau_n\}$ satisfying
\begin{equation}
\label{cs.eq2}\psi_n^\prime(\tau_n)=m_n, \ n\ge1\,,
\end{equation}
\begin{equation}
\label{cs.eq3}a_n^{-1/2}=o(\tau_n), \ n\to\infty\,,
\end{equation}
and such that for all fixed $\delta,\lambda>0$,
\begin{equation}
\label{cs.nonlat}\sup_{\delta\le|t|\le\lambda\tau_n}\left|\frac1{E\left(e^{\tau_n
        T_n}\right)}E\left(e^{(\tau_n+i t)T_n}\right)\right|=o\left(
  a_n^{-1/2}\right), \ n\to\infty\,,
\end{equation}
(with the supremum of the empty set defined as zero). Furthermore,
assume that  
\begin{equation}
\label{cs.eq4} \sup_{n\ge1,\, z\in[-a,a]}|\psi_n(z)|<\infty\ \text{for
  any}\ a>0
\end{equation}
and that 
\begin{equation}
\label{cs.eq5}\inf_{n\ge1}\psi_n^{\prime\prime}(\tau_n)>0\,.
\end{equation}

\medskip

{\bf (a)} \ Under the above assumption, 
\begin{equation}  \label{cs.eq6}
  P\left(a_n^{-1}T_n\ge
  m_n\right)\sim\frac1{\tau_n\sqrt{2\pi
    a_n\psi_n^{\prime\prime}(\tau_n)}}
\exp\left\{-a_n\left(m_n\tau_n-\psi_n(\tau_n)\right)\right\}, \ n\to\infty.
\end{equation}

\medskip

{\bf (b)} \ Let
\[
b_n=\tau_n\sqrt{a_n\psi_n^{\pp}(\tau_n)},
\]
and let $T_n^*$ be a random variable with the law
\[
P(T_n^*\in\,du)=\frac1{E(e^{\tau_nT_n})} e^{u\tau_n}P(T_n\in\,du)\,.
\]
Then  
\begin{equation*}
\sup_{n\ge 1,\, y\in\bbr}{b_n} P\left(y\le\tau_nT^*_n\le y+1\right)
<\infty. 
\end{equation*}
\end{theorem}
\begin{proof}
The first part of the theorem is Theorem 3.3 in
\cite{chaganty:sethuraman:1983}. Furthermore, 
Lemmas 3.1 and 3.2 {\it ibid.} show that the hypotheses (2.7) and
(2.8) of Theorem 2.3  therein hold, and the second part of Theorem \ref{f.cs}
follows from (2.9) of that paper. 
\end{proof}



\end{document}